\documentclass[12pt, a4paper, reqno]{amsart}

  \usepackage{lmodern}

\usepackage{eucal}

\usepackage{hyperref}
\usepackage{amssymb,upref,enumerate}
\usepackage{tikz}
\usepackage{amsmath,amssymb,yfonts,amscd,amsthm,amsxtra,mathrsfs,mathabx}
\usepackage{latexsym}
\usepackage{esint}
\usepackage{mathtools}

\def\XXint#1#2#3{{\setbox0=\hbox{$#1{#2#3}{\int}$} 
\vcenter{\hbox{$#2#3$}}\kern-.5\wd0}}

\newcommand{\ra}{\rightarrow}

\newcommand{\bey}{\begin{eqnarray*}}
\newcommand{\eey}{\end{eqnarray*}}
\newcommand{\ba}{\begin{align}}
\newcommand{\ea}{\end{align}}
\newcommand{\bea}{\begin{align*}}
\newcommand{\ena}{\end{align*}}
\newcommand{\be}{\begin{equation}}
\newcommand{\ee}{\end{equation}}
\newcommand{\R}{\mathbb R}
\newcommand{\Z}{\mathbb Z}

\newcommand{\C}{\mathbb C}

\newcommand{\Sp}{{{\mathbb S}^{d-1}} }
\newcommand{\W}{\textup W }

\newcommand{\ep}{\varepsilon}
\newcommand{\dss}{\displaystyle}
\newcommand{\bc}{\begin{center}}
\newcommand{\ec}{\end{center}}

\newcommand{\vp}{\varphi}

\newcommand{\al}{\alpha}

\newcommand{\dint}{ \ {\mathrm d}}
\newcommand{\dx}{ \ {\mathrm dx}}
\newcommand{\dt}{ \ {\mathrm dt}}
\newcommand{\dy}{ \ {\mathrm dy}}

\newcommand{\Lup}{\textup{L}}
\newcommand{\Int}{\textup{I}}
\newcommand{\M}{\textup{M}}
\newcommand{\T}{\textup{T}}
\newcommand{\Con}{\textup{C}}

\DeclareMathOperator{\Div}{\textup{div}}

\newtheorem{theorem}{Theorem}[section]
\newtheorem{lemma}[theorem]{Lemma}
\newtheorem{corollary}[theorem]{Corollary}
\newtheorem{proposition}[theorem]{Proposition}

\theoremstyle{definition}

\theoremstyle{remark}
\newtheorem{remark}[theorem]{Remark}

\begin{document}




\title[]{Pointwise Riesz potential bounds for rough singular integrals}
\date{\today}

\author{Kabe Moen}

\address{Kabe Moen \\
 Department of Mathematics \\
 University of Alabama \\
 Tuscaloosa, AL 35487, USA}

\email{kabe.moen@ua.edu}

\author{Carlos P\'erez Moreno}

\address{Carlos P\'erez Moreno\\
 Department of Mathematics \\
 University of the Basque Country \\
 Ikerbasque and BCAM \\
 Bilbao, Spain}
 \thanks{C. P\'erez is supported by the Spanish government through grant PID2023-146646NB-I00 and the Severo Ochoa accreditation CEX2021-001142-S, both at BCAM; by the Basque Government through grant IT1875-26 at the University of the Basque Country; and by the BERC program 2022--2025 at BCAM.}

\email{cperez@bcamath.org}

 \begin{abstract} We give a complete characterization of the kernels $\Omega\in \Lup^1(\Sp)$ for which the rough singular integral $\T_\Omega$ and its associated maximal operator $\M_\Omega$ satisfy pointwise bounds in terms of $\Int_1(|\nabla \cdot|)$ and $\M_1(|\nabla \cdot|)$, respectively, where $\Int_1$ is the Riesz potential and $\M_1$ is the fractional maximal operator. We show that these bounds hold if and only if $\Omega$ belongs to the homogeneous dual Sobolev space $\dot \W^{-1,\infty}(\Sp)$. This characterization yields both a PDE characterization and a geometric characterization of the pointwise bounds.  In the Lebesgue scale, we show that, for $d\geq 3$, every rough operator associated to a mean-zero kernel $\Omega\in\Lup^{d-1,\infty}(\Sp)$ satisfies these pointwise bounds, and that the exponent $d-1$ is sharp. In dimension $d=2$, we use complex-analytic methods to establish a new pointwise bound in terms of the complex derivatives $\bar\partial f$ and $\partial f$, which are valid for every mean-zero $\Omega\in\Lup^1(\mathbb S^1)$. As a consequence, we obtain new Sobolev bounds for rough singular integrals in the plane.\end{abstract}
 
\maketitle

\bigskip
\section{Introduction}

For $d\geq 2$, the classical subrepresentation formula states that
\begin{equation}\label{subrep}
|f(x)|\leq c_d\,\Int_1(|\nabla f|)(x),
\qquad f\in \Con_c^\infty(\R^d),\ x\in\R^d,
\end{equation}
where, for $0<\al<d$, $\Int_\al$ is the Riesz potential
\[
\Int_\al f(x)=\int_{\R^d}\frac{f(y)}{|x-y|^{d-\al}}\dy,
\qquad
(\Int_\al f)\hat{\ }(\xi)=c_{\al,d}|\xi|^{-\al}\hat f(\xi).
\]
The associated fractional maximal function
\[
\M_\al f(x)=\sup_{Q\ni x}\ell(Q)^\al\fint_Q |f(y)|\dy,
\qquad 0\leq\al<d,
\]
satisfies $\M_\al f\leq c_{\al,d}\,\Int_\al(|f|)$ pointwise for $\al>0$. For $\al=0$ we write $\M=\M_0$ for the Hardy--Littlewood maximal operator.

The subrepresentation formula \eqref{subrep} is a basic tool in harmonic analysis and PDE, closely tied both to the mapping properties of $\Int_1$ and to Sobolev inequalities (see \cite{KLV}). For example, combining \eqref{subrep} with the Hardy--Littlewood--Sobolev inequality $\Int_1:\Lup^p(\R^d)\ra \Lup^{p^*}(\R^d)$, $1<p<d$, yields the Gagliardo--Nirenberg--Sobolev inequality
\begin{equation}\label{GNS}
\|f\|_{\Lup^{p^*}(\R^d)}\leq C\|\nabla f\|_{\Lup^p(\R^d)},
\qquad 1\leq p<d,\quad p^*=\frac{pd}{d-p},\quad f\in \Con_c^\infty(\R^d).
\end{equation}
At the endpoint $p=1$, $1^*=\frac{d}{d-1}$, the operator $\Int_1$ is only of weak type, $\Int_1:\Lup^1(\R^d)\ra\Lup^{\frac{d}{d-1},\infty}(\R^d)$, so \eqref{subrep} alone gives only the weak-type version of \eqref{GNS}. Nevertheless, the strong inequality involving the gradient \eqref{GNS} with $p=1$ still holds; it can be recovered from the weak-type estimate by the truncation method of Maz'ya (see \cite{Mazya}). Moreover, it was shown in \cite{BMOPR} that \eqref{subrep} is equivalent to an endpoint version of \eqref{GNS} with respect to general measures.

Let $\dot \W^{1,p}(\R^d)$ denote the homogeneous Sobolev space, defined as the completion of $\Con_c^\infty(\R^d)$ with respect to the norm $\|\nabla f\|_{\Lup^p(\R^d)}$. For $1\leq p<d$, \eqref{GNS} yields the Sobolev embedding $\dot \W^{1,p}(\R^d)\hookrightarrow \Lup^{p^*}(\R^d)$ and allows us to identify this space with
\[
\dot \W^{1,p}(\R^d)
=
\left\{
f\in \Lup^{p^*}(\R^d):
\|\nabla f\|_{\Lup^p(\R^d)}<\infty
\right\},
\qquad
\|f\|_{\dot \W^{1,p}(\R^d)}
=
\|\nabla f\|_{\Lup^p(\R^d)},
\]
where $\nabla f$ denotes the distributional gradient.

In \cite{HMP,HMP2}, the subrepresentation formula \eqref{subrep} was extended to operators:
\begin{equation}\label{opersubrep}
|\T f(x)|\leq c_\T\,\Int_1(|\nabla f|)(x),
\qquad f\in \Con_c^\infty(\R^d),\ x\in\R^d,
\end{equation}
where $\T$ is defined on $\Con_c^\infty(\R^d)$. The subrepresentation formula \eqref{subrep} corresponds to $\T=\textup{Id}$, but many classical operators in harmonic analysis also satisfy \eqref{opersubrep}. For example, the $\Lup^r$ variant of the Hardy--Littlewood maximal operator
\[
\M_{\Lup^r}f(x)=\sup_{Q\ni x}\Big(\fint_Q |f(y)|^r\dy\Big)^{1/r},
\qquad r>0,
\]
satisfies \eqref{opersubrep} if and only if $0<r\leq \frac{d}{d-1}$ \cite{HMP,HMP2}. Maximal operators along surfaces also satisfy \eqref{opersubrep}; for instance, it is shown in \cite{HMP2} that the spherical maximal operator
\[
\M_{\Sp}f(x)=\sup_{t>0}\fint_{\partial B(x,t)}|f|\,\dint\mathcal H^{d-1}
\]
satisfies $\M_{\Sp}f\leq c_d\,\Int_1(|\nabla f|)$ (see also \cite{HajlaszLiu}).

Lerner \cite{Lern} recently characterized the Calder\'on--Zygmund operators that satisfy \eqref{opersubrep}. Roughly speaking, a Calder\'on--Zygmund operator is an operator bounded on $\Lup^2(\R^d)$ that is given, away from the diagonal, by integration against a kernel satisfying a size and smoothness condition; see \cite{Gra14,Duo01} for the precise definition. Such operators extend to bounded operators from $\Lup^\infty(\R^d)$ into $\textup{BMO}$, so in particular $\T1$ is well defined as an element of $\textup{BMO}$.
\begin{theorem}[\cite{Lern}] Suppose $\T$ is a Calder\'on--Zygmund operator. Then
\begin{equation}\label{CZOchar}
|\T f(x)|\leq c_\T\,\Int_1(|\nabla f|)(x),\qquad f\in \Con_c^\infty(\R^d),\ x\in\R^d,
\end{equation}
if and only if $\T1\in \Lup^\infty(\R^d)$.
\end{theorem}
By the $\T1$ theorem of David and Journ\'e \cite{DJ}, an operator with such a kernel is bounded on $\Lup^2(\R^d)$ if and only if $\T1,\T^*1\in\textup{BMO}$ and $\T$ has the weak boundedness property. Thus \eqref{CZOchar} singles out those Calder\'on--Zygmund operators for which $\T1$ lies in the subspace $\Lup^\infty(\R^d)$ of $\textup{BMO}$.

Our main interest will be rough operators associated to $\Omega\in\Lup^1(\Sp)$ with mean zero. Writing $y'=y/|y|$, for $f\in \Con_c^\infty(\R^d)$ define
\[
\T_\Omega f(x)=\textup{p.v.}\int_{\R^d}\frac{\Omega(y')}{|y|^d}\,f(x-y)\dy=\lim_{\ep\ra 0^+}\int_{|y|>\ep}\frac{\Omega(y')}{|y|^d}\,f(x-y)\dy,
\]
and
\[
\M_\Omega f(x)=\sup_{r>0}\left|\fint_{B(0,r)} \Omega(y')f(x-y)\dy\right|.
\]
The principal value defining $\T_\Omega f$ exists for every $f\in \Con_c^\infty(\R^d)$ provided $\int_\Sp\Omega=0$; without this condition, the truncated integrals diverge like $\log(1/\ep)\int_\Sp\Omega$ (see the proof of Theorem~\ref{maincharthm}). By contrast, $\M_\Omega f$ is finite for every $f\in \Con_c^\infty(\R^d)$ and every $\Omega\in \Lup^1(\Sp)$, but the pointwise bounds below again require $\Omega$ to have mean zero. When $\int_\Sp\Omega=0$, the operator $\M_\Omega$ has cancellation; in particular, it is not the maximal operator defined with respect to $|\Omega|$.

The operators $\T_\Omega$ and $\M_\Omega$ have a long and rich history, for which we refer the reader to the books of Grafakos and Duoandikoetxea \cite{Gra14,Duo01}. Classical results on the boundedness of $\T_\Omega$ and $\M_\Omega$ can be found in \cite{CZ56,Chr88,CRdF88,Hof88,See96,GS98}, and these operators remain an active object of study \cite{HRT17,CACDPO17,Ler19,LPRR19,DPHL20,Lai25}. One of the difficulties is that their kernels lack smoothness, so they do not fit into the realm of Calder\'on--Zygmund operators. In particular, rough singular integrals are not covered by \eqref{CZOchar}; note also that $\T_\Omega 1=0$.

In \cite{HMP} it was shown that if $\Omega\in\Lup^{d,\infty}(\Sp)$ has mean zero, then
\begin{equation}\label{ptwiserough}
|\T_\Omega f(x)|\leq c_d\|\Omega\|_{\Lup^{d,\infty}(\Sp)}\,\Int_1(|\nabla f|)(x),
\qquad f\in \Con_c^\infty(\R^d),\ x\in \R^d.
\end{equation}
The maximal operator $\M_\Omega$ was later studied in \cite{HMP2}, where it was shown that, under the same assumption,
\begin{equation}\label{ptwiseroughmax}
\M_\Omega f(x)\leq c_d\|\Omega\|_{\Lup^{d,\infty}(\Sp)}\,\M_1(|\nabla f|)(x),
\qquad f\in \Con_c^\infty(\R^d),\ x\in \R^d.
\end{equation}
The case $\Omega\in\Lup^r(\Sp)$ with $1<r<d$ is treated in \cite{HMP2}, where pointwise bounds with a different right-hand side are obtained by means of sparse fractional operators adapted to $\Lup^s$ averages, $s>1$. We also refer the reader to \cite{CMM,CMMCor}, where bounds for $\Omega\in\Lup^r(\Sp)$, $1<r<d$, are established with a right-hand side involving Morrey spaces; see also \cite{CMM2} for further developments in this vein.
\subsection*{Main results}

The purpose of this work is to characterize the kernels $\Omega$ for which the pointwise inequalities \eqref{ptwiserough} and \eqref{ptwiseroughmax} hold. Our first result gives an improved sufficient condition that lowers the integrability index from $d$ to $d-1$, and we show that $d-1$ is sharp on the Lebesgue scale.

\begin{theorem}\label{mainthm}
Suppose $d\geq 3$ and $\Omega\in\Lup^{d-1,\infty}(\Sp)$ has mean zero. Then
\begin{equation}\label{newptwiserough}
|\T_\Omega f(x)|\leq c_d\|\Omega\|_{\Lup^{d-1,\infty}(\Sp)}\,\Int_1(|\nabla f|)(x),
\end{equation}
and
\begin{equation}\label{newptwisemax}
\M_\Omega f(x)\leq c_d\|\Omega\|_{\Lup^{d-1,\infty}(\Sp)}\,\M_1(|\nabla f|)(x),
\end{equation}
where both inequalities hold for $f\in \Con_c^\infty(\R^d)$ and $x\in \R^d$. The result is sharp: for each $1\leq r<d-1$ there exist a mean-zero $\Omega\in\Lup^r(\Sp)$ and functions $f_\delta\in \Con_c^\infty(\R^d)$, $0<\delta\ll 1$, such that
\[
\lim_{\delta\to 0^+}\frac{|\T_\Omega f_\delta(0)|}{\Int_1(|\nabla f_\delta|)(0)}=\infty \qquad \text{and} \qquad \lim_{\delta\to 0^+}\frac{\M_\Omega f_\delta(0)}{\M_1(|\nabla f_\delta|)(0)}=\infty;
\]
in particular, \eqref{newptwiserough} and \eqref{newptwisemax} fail for such $\Omega$.
\end{theorem}

\begin{remark}
The proof of Theorem~\ref{mainthm} gives more: it controls the maximal truncated operator
\[
  \T^\star_\Omega f(x)
  = \sup_{\ep>0} \bigl|\T^\ep_\Omega f(x)\bigr|
  = \sup_{\ep>0} \left| \int_{|y|>\ep} \frac{\Omega(y')}{|y|^d}\, f(x-y)\dy \right|
\]
pointwise, namely
\[
  \T^\star_\Omega f(x) \le c_d\|\Omega\|_{\Lup^{d-1,\infty}(\Sp)} \Int_1\bigl(|\nabla f|\bigr)(x).
\]
\end{remark}

The main new ingredient in the proof of Theorem~\ref{mainthm} is a Poincar\'e inequality on the sphere. For $d\geq3$, the Lorentz--Poincar\'e inequality on the $(d-1)$-dimensional sphere (Lemma~\ref{lem:lorentzpoincare}) states that
\[
\|\vp-\vp_{\Sp}\|_{\Lup^{\frac{d-1}{d-2},1}(\Sp)}\leq C_d\int_{\Sp}|\nabla_\Sp\vp|\,\dint\mathcal H^{d-1},
\qquad \vp_\Sp=\fint_\Sp\vp\,\dint\mathcal H^{d-1},
\]
where $\nabla_\Sp$ is the \emph{spherical gradient} (see Section~\ref{prelim}). By contrast, the proof in \cite{HMP} relies on the Lorentz--Poincar\'e inequality on cubes in $\R^d$, which is originally due to Alvino \cite{Alvino},
\[
\|f-f_Q\|_{\Lup^{\frac{d}{d-1},1}(Q,\frac{\dx}{|Q|})}\leq C_d\,\ell(Q)\fint_Q|\nabla f(x)|\dx.
\]
The gain in the exponent in Theorem~\ref{mainthm} reflects the drop in dimension: the cancellation of $\Omega$ is exploited only in the angular variable.

For $d=2$, every mean-zero $\Omega\in\Lup^1(\mathbb S^1)$ gives rise to a pointwise inequality. We prove the planar case with complex-analytic methods, which allow us to work with complex-valued $f\in\Con^\infty_c(\C)$ and the derivatives $\bar{\partial}=\frac12(\partial_x+i\partial_y)$ and ${\partial}=\frac12(\partial_x-i\partial_y)$.

\begin{theorem}\label{mainthmplane}
Suppose $\Omega\in\Lup^{1}(\mathbb S^1)$ has mean zero. Then
\begin{equation}\label{newptwiseroughplane}
|\T_\Omega f(z)|\leq \|\Omega\|_{\Lup^{1}(\mathbb S^1)}\,\min\big\{\Int_1(|\bar{\partial}f|)(z),\,\Int_1(|\partial f|)(z)\big\},
\qquad f\in \Con_c^\infty(\C),\ z\in \C.
\end{equation}
\end{theorem}

The proof realizes $\T_\Omega$ as a twisted Cauchy transform (Proposition~\ref{twistCauchyprop}). For real-valued $f$, \eqref{newptwiseroughplane} is a subrepresentation formula with the gradient, since then $|\bar\partial f|=|\partial f|=\frac12|\nabla f|$. For complex-valued $f$ we measure the total derivative $Df$ in the operator norm, so that $|Df|=|\partial f|+|\bar\partial f|$, and let $\dot\W^{1,p}(\C)$ be the completion of $\Con_c^\infty(\C)$ with respect to $\|Df\|_{\Lup^p(\C)}$.

Theorem~\ref{mainthmplane} yields Sobolev bounds that do not follow from Lebesgue-space bounds. Whenever $\T_\Omega$ is bounded on $\Lup^{p^*}(\R^d)$, a Sobolev bound is immediate from \eqref{GNS}:
\[
\|\T_\Omega f\|_{\Lup^{p^*}(\R^d)}\leq C\|f\|_{\Lup^{p^*}(\R^d)}\leq C'\|\nabla f\|_{\Lup^p(\R^d)},\qquad 1<p<d.
\]
This is the situation in Theorem~\ref{mainthm}: for $d\geq3$, the condition $\Omega\in\Lup^{d-1,\infty}(\Sp)$ is strong enough to guarantee that $\T_\Omega$ is bounded on $\Lup^p(\R^d)$ for all $1<p<\infty$ (see \cite[Thm.~4.12]{Duo01}), so \eqref{newptwiserough} gives no new norm inequalities.

For general $\Omega\in\Lup^1(\Sp)$, however, $\T_\Omega$ need not be bounded on any $\Lup^p(\R^d)$. Weiss and Zygmund \cite{WZ65} constructed a mean-zero $\Omega\in\Lup^1(\Sp)$ for which $\T_\Omega$ is unbounded on $\Lup^2(\R^d)$, and hence, by duality and interpolation, unbounded on $\Lup^p(\R^d)$, for any $1<p<\infty$ (see also \cite[Thm.~1]{GS99}). Even in the plane, the situation is delicate: boundedness of $\T_\Omega$ on some $\Lup^p(\R^2)$ does not necessarily extend beyond a neighborhood of $p=2$ \cite{Hon11}, and $\T_\Omega$ may be bounded on $\Lup^2(\R^2)$ while $\T^\star_\Omega$ is unbounded \cite{Hon10}.

This is what makes Theorem \ref{mainthmplane} in the planar case special. For general mean-zero $\Omega\in\Lup^1(\mathbb S^1)$, Sobolev bounds cannot be obtained by combining Lebesgue-space bounds with \eqref{GNS}, yet the pointwise inequality \eqref{newptwiseroughplane} bypasses the Lebesgue theory entirely and yields genuinely new Sobolev bounds for \emph{every} such $\Omega$.

\begin{corollary}\label{boundedplane}
Suppose $\Omega\in\Lup^{1}(\mathbb S^1)$ has mean zero and $1<p<2$. Then
\begin{equation}\label{LpSob}
\|\T_\Omega f\|_{\Lup^{\frac{2p}{2-p}}(\C)}\leq C_p\|\Omega\|_{\Lup^1(\mathbb S^1)}\|Df\|_{\Lup^p(\C)},\qquad f\in\Con_c^\infty(\C),
\end{equation}
so $\T_\Omega$ extends to a bounded operator $\dot\W^{1,p}(\C)\ra\Lup^{\frac{2p}{2-p}}(\C)$.
\end{corollary}

Nothing is lost in stating \eqref{LpSob} with $Df$: for $1<p<\infty$, $\|\bar\partial f\|_{\Lup^p(\C)}\simeq\|\partial f\|_{\Lup^p(\C)}\simeq\|Df\|_{\Lup^p(\C)}$. This follows from the identities $\mathcal{B}(\bar\partial f)=\partial f$ and $\mathcal B^{-1}(\partial f)=\bar\partial f$ for the Beurling--Ahlfors transform
\[
\mathcal{B}f(z)=-\frac1\pi\,\textup{p.v.}\int_{\C}\frac{f(w)}{(z-w)^2}\,\dint A(w),
\]
which is $\T_\Omega$ with $\Omega(\zeta)=-\bar\zeta^{2}/\pi$ and is bounded on $\Lup^p(\C)$ (see \cite{AABetal26} for the sharp norm), together with $\mathcal B^{-1}=\mathcal B^*$.

For $p=1$ the situation is different: $\mathcal B$ is only of weak type $(1,1)$, and neither of $\|\partial f\|_{\Lup^1(\C)}$ and $\|\bar\partial f\|_{\Lup^1(\C)}$ is controlled by the other \cite{Orn62}. Let $\dot\W^{1,1}_{\bar\partial}(\C)$ and $\dot\W^{1,1}_{\partial}(\C)$ be the completions of $\Con_c^\infty(\C)$ with respect to $\|\bar\partial f\|_{\Lup^1(\C)}$ and $\|\partial f\|_{\Lup^1(\C)}$. Since $\|Df\|_{\Lup^1(\C)}=\|\bar\partial f\|_{\Lup^1(\C)}+\|\partial f\|_{\Lup^1(\C)}$, we have $\dot\W^{1,1}(\C)=\dot\W^{1,1}_{\bar\partial}(\C)\cap\dot\W^{1,1}_{\partial}(\C)$, and neither of the two spaces contains the other. In particular, $\dot\W^{1,1}(\C)$ is a proper subspace of each of them, and both are proper subspaces of $\dot\W^{1,1}_{\bar\partial}(\C)+\dot\W^{1,1}_{\partial}(\C)$ (see Section \ref{plane}). At the endpoint, \eqref{newptwiseroughplane} gives a bound on this larger space.

\begin{corollary}\label{boundedplaneweak}
Suppose $\Omega\in\Lup^{1}(\mathbb S^1)$ has mean zero. Then
\begin{equation}\label{weakSob}
\|\T_\Omega f\|_{\Lup^{2,\infty}(\C)}\leq C\|\Omega\|_{\Lup^1(\mathbb S^1)}\min\big\{\|\bar\partial f\|_{\Lup^1(\C)},\,\|\partial f\|_{\Lup^1(\C)}\big\},\qquad f\in\Con_c^\infty(\C),
\end{equation}
so $\T_\Omega$ extends to a bounded operator $\dot\W^{1,1}_{\bar\partial}(\C)+\dot\W^{1,1}_{\partial}(\C)\ra\Lup^{2,\infty}(\C)$.
\end{corollary}

\begin{remark} The space $\Lup^{2,\infty}(\C)$ in \eqref{weakSob} cannot be replaced by $\Lup^2(\C)$, even for the identity. By the Cauchy--Pompeiu formula (see \cite{AIM09}) and its conjugate,
\[
f(z)=\frac1\pi\int_{\C}\frac{\bar\partial f(w)}{z-w}\,\dint A(w)=\frac1\pi\int_{\C}\frac{\partial f(w)}{\overline{z-w}}\,\dint A(w),\qquad f\in\Con_c^\infty(\C),
\]
so $|f|\leq\frac1\pi\min\{\Int_1(|\bar\partial f|),\Int_1(|\partial f|)\}$, and hence $\dot\W^{1,1}_{\bar\partial}(\C)+\dot\W^{1,1}_\partial(\C)\hookrightarrow\Lup^{2,\infty}(\C)$. However, while the Gagliardo--Nirenberg inequality gives $\|f\|_{\Lup^2(\C)}\leq C\|Df\|_{\Lup^1(\C)}$, the estimates $\|f\|_{\Lup^2(\C)}\leq C\|\bar\partial f\|_{\Lup^1(\C)}$ and $\|f\|_{\Lup^2(\C)}\leq C\|\partial f\|_{\Lup^1(\C)}$ both fail. This is an instance of Van Schaftingen's theorem \cite{VS13}: for an elliptic first-order operator $A(D)$, the estimate $\|f\|_{\Lup^{\frac{n}{n-1}}(\R^n)}\leq C\|A(D)f\|_{\Lup^1(\R^n)}$ holds if and only if $A(D)$ is \emph{canceling}. The full derivative $D=(\partial,\bar\partial)$ is canceling, while $\partial$ and $\bar\partial$ alone are not.
\end{remark}

There is a complete characterization of the pointwise inequalities in terms of a dual Sobolev space on the sphere. For $\Omega\in\Lup^1(\Sp)$ we say $\Omega \in \dot \W^{-1,\infty}(\Sp)$ if
\begin{multline}\label{dualnorm}
\|\Omega\|_{\dot \W^{-1,\infty}(\Sp)}\\=\sup\left\{\Big|\int_\Sp\Omega\,\vp\,\dint\mathcal H^{d-1}\Big|:\ \vp\in \Con^\infty(\Sp),\ \|\nabla_\Sp\vp\|_{\Lup^1(\Sp)}\leq 1\right\}<\infty.
\end{multline}
That is, \eqref{dualnorm} is the norm of $\Omega$ as a linear functional on $\dot \W^{1,1}(\Sp)$. Given a vector field $ \mathbf{V}$ on $\Sp$, we define the \emph{spherical divergence} $\Div_\Sp  \mathbf{V}$ via the weak formulation 
\begin{equation}\label{weakdiv}
\int_\Sp \Div_\Sp  \mathbf{V}\,\vp\dint\mathcal H^{d-1}=-\int_\Sp  \mathbf{V}\cdot \nabla_\Sp \vp \dint\mathcal H^{d-1}\qquad\text{for all } \vp\in\Con^\infty(\Sp)
\end{equation} (see Section~\ref{prelim} for more on the $\Div_{\Sp}$ operator).

\begin{theorem}\label{maincharthm}
Suppose $d\geq 2$ and $\Omega\in\Lup^1(\Sp)$. The following are equivalent:
\begin{enumerate}
\item $\Omega\in \dot \W^{-1,\infty}(\Sp)$;
\item $\M_\Omega f(x)\leq c\,\M_1(|\nabla f|)(x)$ for all $f\in \Con_c^\infty(\R^d)$ and $x\in \R^d$;
\item $|\T_\Omega f(x)|\leq c\,\Int_1(|\nabla f|)(x)$ for all $f\in \Con_c^\infty(\R^d)$ and $x\in \R^d$;
\item $\dss\Big|\int_E \Omega\dint\mathcal H^{d-1}\Big|\leq C\,\mathcal H^{d-2}(\partial E)$ for all open $E\subseteq \Sp$ with smooth boundary;
\item $\Omega=\Div_\Sp \mathbf{V}$ for some $ \mathbf{V}\in \Lup^\infty(\Sp,\R^d)$.
\end{enumerate}
Moreover, the best constants in (2) and (3) are comparable to $\|\Omega\|_{\dot \W^{-1,\infty}(\Sp)}$, with constants depending only on $d$, the best constant in (4) equals $\|\Omega\|_{\dot \W^{-1,\infty}(\Sp)}$, and
\[
\|\Omega\|_{\dot \W^{-1,\infty}(\Sp)}=\min\big\{\| \mathbf{V}\|_{\Lup^\infty(\Sp,\R^d)}:\ \Omega=\Div_\Sp \mathbf{V}\big\}.
\]
\end{theorem}

\begin{remark}
We do not assume that $\Omega$ has mean zero in Theorem~\ref{maincharthm}: each of the conditions (1)--(5) individually forces $\int_{\Sp}\Omega\dint\mathcal H^{d-1} = 0$ (see the first step of the proof). In particular, none of them can hold when $\int_\Sp\Omega\neq0$. For (5) this relies on the weak definition \eqref{weakdiv} of $\Div_\Sp$, in which only the tangential part $\mathsf P_\theta\mathbf V$ of $\mathbf V$ enters.
\end{remark}

Conditions (4) and (5) are spherical analogues of known results for the divergence equation $\Div \mathbf{Y}=F$ in Euclidean space. The solvability of $\Div \mathbf{Y}=F$ with $\mathbf Y\in\Lup^\infty$ was studied by Bourgain and Brezis \cite{BB03,BB07}; see also \cite{BMR,DPP,VS14}. By duality, a solution exists exactly when $F$ defines a bounded functional on $\dot\W^{1,1}$; this is the Euclidean analogue of the equivalence (1)$\Leftrightarrow$(5). Phuc and Torres \cite{PT} characterized the existence of bounded solutions in terms of the isoperimetric-type condition $|F(U)|\leq C\,\mathcal H^{d-1}(\partial U)$ for smooth bounded open sets $U$; condition (4) is the analogue of this on $\Sp$.  

\subsection*{Organization}
In Section~\ref{prelim} we collect the preliminaries: Hausdorff measure, the spherical gradient, the weak spherical divergence, Poincar\'e inequalities on the sphere, including Lemma~\ref{lem:lorentzpoincare}, and some facts from complex analysis. In Section~\ref{proofs} we prove Theorem~\ref{mainthm}, including its sharpness, and the characterization Theorem~\ref{maincharthm}. In Section~\ref{plane} we treat the planar case and prove Theorem~\ref{mainthmplane} and Corollaries~\ref{boundedplane} and~\ref{boundedplaneweak}.

\section{Preliminaries}\label{prelim}
In this section we collect the background material used in our proofs.  These topics include, Lorentz spaces, Hausdorff measure, geometry on the sphere, and some facts from complex analysis.  

\subsection{Lorentz spaces}\label{sec:lorentz}

We recall a few basic facts about Lorentz spaces; see, e.g., \cite[\S1.4]{Gra14}.
Let $(X,\mu)$ be a measure space, and for a measurable function $f$ on $X$ and $1\leq p,q<\infty$ the
Lorentz space $\Lup^{p,q}(\mu)$ consists of all measurable $f$ with
\[
\|f\|_{\Lup^{p,q}(\mu)}^q=p\int_0^\infty t^q\,\mu(\{x:|f(x)|>t\})^{\frac{q}{p}}\,\frac{\dt}{t}<\infty,
\]
and for $q=\infty$ it consists of all measurable $f$ with
\[
\|f\|_{\Lup^{p,\infty}(\mu)}=\sup_{t>0}\,t\,\mu(\{x:|f(x)|>t\})^{\frac1p}<\infty.
\]
These are quasi-norms, $\|f\|_{\Lup^{p,p}(\mu)}=\|f\|_{\Lup^p(\mu)}$, and
$\Lup^{p,q_1}(\mu)\subseteq\Lup^{p,q_2}(\mu)$ continuously when $q_1\leq q_2$. For $1<p<\infty$, $1\leq q\leq\infty$, the
Lorentz--H\"older inequality
\begin{equation}\label{lorentzholder}
\int_X|fg|\dint\mu\leq \|f\|_{\Lup^{p,q}(\mu)}\,\|g\|_{\Lup^{p',q'}(\mu)}
\end{equation}
holds where $p'$ and $q'$ are the usual dual
exponents. 

\subsection{Hausdorff measure}\label{sec:hausdorff}

Following \cite[Chap. 2]{EG}, for $0\leq s<\infty$, the $s$-dimensional Hausdorff measure of a set $A\subseteq\R^d$ is
\[
\mathcal H^s(A)=\lim_{\delta\to0^+}\inf\Big\{\sum_j \alpha_s\Big(\frac{\textup{diam}\,C_j}{2}\Big)^s:\ A\subseteq\bigcup_j C_j,\ \textup{diam}\,C_j<\delta\Big\},
\qquad \alpha_s=\frac{\pi^{s/2}}{\Gamma(\frac s2+1)},
\]
normalized so that $\mathcal H^d$ coincides with Lebesgue measure on $\R^d$. We use $\mathcal H^s$ only on smooth submanifolds, where it agrees with the usual Riemannian volume measure. In particular:
\begin{enumerate}
\item $\mathcal H^{d-1}$ restricted to $\Sp$ (or to $\partial B(x,t)$) is the standard surface measure, $\omega_{d-1}=\mathcal H^{d-1}(\Sp)$ is the surface area of the unit sphere, and integration in polar coordinates reads
\[
\int_{\R^d}F(y)\dy=\int_0^\infty\int_\Sp F(\rho\theta)\dint\mathcal H^{d-1}(\theta)\,\rho^{d-1}\dint\rho,\qquad F\in\Lup^1(\R^d).
\]
\item If $E\subseteq\Sp$ is open with smooth boundary, then $\partial E$ is a compact smooth $(d-2)$-dimensional submanifold of $\Sp$, and $\mathcal H^{d-2}(\partial E)$ is its $(d-2)$-dimensional volume. When $d=2$, $\mathcal H^0$ is counting measure, so $\mathcal H^0(\partial E)$ is the number of endpoints of the finitely many arcs that make up $E$.
\end{enumerate}

\subsection{Geometry on $\Sp$}\label{sec:sphgrad}

In this section, we review the basic facts about analysis on the unit sphere that are used throughout the paper. We begin by defining the spherical gradient and divergence and relating them to their Euclidean counterparts. We then record the Poincar\'e inequalities on $\Sp$ that
we need. Finally, we discuss smooth subsets of $\Sp$, the coarea
formula, and the signed distance to a smooth boundary, which are used in
the proof of the sharpness of the constant in~(4).

\subsubsection{The spherical gradient and divergence}

For $\vp\in\Con^1(\Sp)$, let $\Phi(y)=\vp(y/|y|)$ be its $0$-homogeneous
extension to $\R^d\setminus\{0\}$, and define
\[
\nabla_\Sp\vp(\theta)=\nabla\Phi(\theta),\qquad \theta\in\Sp.
\]
Differentiating $\Phi(ty)=\Phi(y)$ at $t=1$ gives $y\cdot\nabla\Phi(y)=0$
for $y\neq0$. In particular $\nabla_\Sp\vp(\theta)\cdot\theta=0$, so
$\nabla_\Sp\vp(\theta)$ is tangent to $\Sp$ at $\theta$.

Let $\mathsf P_\theta=\textup{Id}-\theta\theta^{\mathsf T}$ be the
orthogonal projection onto $\theta^\perp$, so that
$\mathsf P_\theta w=w-(w\cdot\theta)\theta$ and $|\mathsf P_\theta w|\leq|w|$.
If $\Psi$ is any $\Con^1$ function on a neighborhood of $\Sp$ with
$\Psi|_\Sp=\vp$, then $\Psi$ and $\Phi$ have the same tangential
derivatives on $\Sp$, and therefore
\[
\nabla_\Sp\vp(\theta)=\mathsf P_\theta\nabla\Psi(\theta),\qquad \theta\in\Sp.
\]
Equivalently, $\nabla_\Sp\vp=(\delta_1\vp,\dots,\delta_d\vp)$, where
$\delta_j=\partial_j-\theta_j\,(\theta\cdot\nabla)$ is applied to any such
extension.

The spherical gradient is the angular part of the Euclidean gradient. For
$F\in\Con^1(\R^d\setminus\{0\})$ and $F_\rho(\theta)=F(\rho\theta)$,
\begin{equation}\label{polargrad}
\nabla F(\rho\theta)=\theta\,\partial_\rho F(\rho\theta)+\frac1\rho\,\nabla_\Sp F_\rho(\theta),
\end{equation}
so that
$|\nabla F(\rho\theta)|^2=|\partial_\rho F(\rho\theta)|^2+\rho^{-2}|\nabla_\Sp F_\rho(\theta)|^2$.

For $\mathbf V\in\Lup^1(\Sp,\R^d)$, the spherical divergence
$\Div_\Sp\mathbf V$ is defined weakly by \eqref{weakdiv}. Since
$\nabla_\Sp\vp(\theta)\perp\theta$, only the tangential part of
$\mathbf V$ enters \eqref{weakdiv}. Hence
$\Div_\Sp\mathbf V=\Div_\Sp(\mathsf P_\theta\mathbf V)$, where
$|\mathsf P_\theta\mathbf V|\leq|\mathbf V|$. If $\mathbf V$ is smooth, then
\[
\Div_\Sp\mathbf V=\sum_{j=1}^d\delta_j(\mathsf P_\theta\mathbf V)_j
=\sum_{j=1}^d\delta_jV_j-(d-1)\,\mathbf V\cdot\theta,
\]
where the last term comes from $\sum_{j=1}^d\delta_j\theta_j=d-1$ (and
$\sum_j\theta_j\delta_j=0$). In particular,
$\Div_\Sp\mathbf V=\sum_{j=1}^d\delta_jV_j$ when $\mathbf V$ is tangential,
that is, when $\mathbf V(\theta)\cdot\theta=0$ for $\theta\in\Sp$.

\subsubsection{Poincar\'e inequalities on $\Sp$}

For background on Poincar\'e inequalities we refer to the books of
Maz'ya \cite{Mazya} and Saloff-Coste \cite{SC}; the latter works in the
setting of Riemannian manifolds. We denote the average of $f$ over $\Sp$ by
\[
f_\Sp=\fint_{\Sp}f\dint\mathcal H^{d-1}=\frac{1}{\omega_{d-1}}\int_\Sp f\dint\mathcal H^{d-1},
\]
We need the following Lorentz-space Poincaré inequality, which follows from
\cite[Theorem~6 and Example~4]{MM}.

\begin{lemma}\label{lem:lorentzpoincare}
Let $d\geq3$. There is a constant $C=C(d)$ such that for all
$f\in\Con^1(\Sp)$,
\[
\|f-f_\Sp\|_{\Lup^{\frac{d-1}{d-2},1}(\Sp)}\leq C\|\nabla_\Sp f\|_{\Lup^1(\Sp)}.
\]
\end{lemma}

We also use the $\Lup^1$ Poincar\'e inequality
\begin{equation}\label{poincareL1}
\|\vp-\vp_\Sp\|_{\Lup^1(\Sp)}\leq C_d\|\nabla_\Sp\vp\|_{\Lup^1(\Sp)},\qquad \vp\in\Con^1(\Sp).
\end{equation}
For $d\geq3$ this follows from Lemma~\ref{lem:lorentzpoincare} and
\eqref{lorentzholder} which imply $\Lup^{\frac{d-1}{d-2},1}(\Sp)\subseteq \Lup^1(\Sp)$. For $d=2$ it is elementary. Write $\vp(t)$ for
$\vp(e^{it})$, so that $|\nabla_{\mathbb S^1}\vp|=|\vp'|$. The fundamental
theorem of calculus gives $|\vp(t)-\vp(t')|\leq\|\vp'\|_{\Lup^1(\mathbb S^1)}$,
and averaging in $t'$ yields
\[
\|\vp-\vp_{\mathbb S^1}\|_{\Lup^\infty(\mathbb S^1)}\leq\|\nabla_{\mathbb S^1}\vp\|_{\Lup^1(\mathbb S^1)},
\]
which implies \eqref{poincareL1}; see also the proof of
Lemma~\ref{primitiveV}.

\subsubsection{Coarea formula on $\Sp$ and signed distance}\label{sec:signeddist}

We use the coarea formula on $\Sp$ in the following form; see
\cite[\S3.2]{Federer}. If $U\subseteq\Sp$ is open, $u\in\Con^1(U)$, and
$g\colon U\to[0,\infty]$ is Borel, then
\begin{equation}\label{coarea}
\int_U g\,|\nabla_\Sp u|\dint\mathcal H^{d-1}
=\int_{-\infty}^{\infty}\int_{\{u=t\}}g\dint\mathcal H^{d-2}\dint t.
\end{equation}
In particular, for $\vp\in\Con^\infty(\Sp)$,
\[
\int_\Sp|\nabla_\Sp\vp|\dint\mathcal H^{d-1}=\int_{-\infty}^{\infty}\mathcal H^{d-2}(\{\vp=t\})\dint t.
\]
By Sard's theorem \cite{LeeSmooth}, almost every $t$ is a regular value
of $\vp$. For such $t$, the level set $\{\vp=t\}$ is a smooth
$(d-2)$-dimensional submanifold of $\Sp$, and it is the boundary of the
open set $\{\vp>t\}$.

We write $\operatorname{dist}_\Sp(\theta,\eta)=\arccos(\theta\cdot\eta)$
for the geodesic distance on $\Sp$. Let $E\subseteq\Sp$ be open with
smooth boundary. The \emph{signed geodesic distance} to $\partial E$ is
\begin{equation}\label{geodes}
h(\theta)=\pm\operatorname{dist}_\Sp(\theta,\partial E),
\end{equation}
with the plus sign for $\theta\in E$ and the minus sign otherwise. It is
Lipschitz with respect to $\operatorname{dist}_\Sp$. For example, for
a spherical cap $E=\{\theta:\theta\cdot e>\cos r\}$ we have
$h(\theta)=r-\arccos(\theta\cdot e)$.

\begin{lemma}\label{lem:signed-distance}
There is $t_0>0$ such that $h$ is smooth with $|\nabla_\Sp h|=1$ on
$\{|h|<t_0\}$, and $\mathcal H^{d-2}(\{h=t\})\to\mathcal H^{d-2}(\partial E)$
as $t\to0$.
\end{lemma}

These facts are standard; see, e.g., \cite{Gray,MaMe}, and
\cite[Lemma~14.16]{GT} and \cite{KP} for the Euclidean case. Briefly, let
$\nu(p)$ denote the unit normal to $\partial E$ at $p$ that is tangent to
$\Sp$ and points into $E$. By the tubular neighborhood theorem \cite{Lee}
there is $t_0>0$ such that $\Gamma(p,s)=\cos s\,p+\sin s\,\nu(p)$ is a
diffeomorphism from $\partial E\times(-t_0,t_0)$ onto $\{|h|<t_0\}$ with
$h(\Gamma(p,s))=s$. Hence $h$ is smooth on $\{|h|<t_0\}$, with
$|\nabla_\Sp h|=1$ because $h$ increases at unit rate along the unit-speed
geodesics $s\mapsto\Gamma(p,s)$, which cross its level sets orthogonally.
Moreover, $\{h=t\}=\Gamma(\partial E,t)$, and $\Gamma(\cdot,t)$ depends
smoothly on $t$ and is the identity at $t=0$, so
$\mathcal H^{d-2}(\{h=t\})\to\mathcal H^{d-2}(\partial E)$.

\subsection{Complex analysis preliminaries}

Identifying $\C$ with $\R^2$ by writing $w=x+iy$, recall the Wirtinger derivatives
\[
\partial=\tfrac12(\partial_x-i\partial_y),\qquad
\bar\partial=\tfrac12(\partial_x+i\partial_y).
\]
Both are linear and satisfy the product rule, and
$\overline{\bar\partial f}=\partial\bar f$. One has $\bar\partial w=0$ and
$\bar\partial\bar w=1$, and a $\Con^1$ function $g$ on an open set is holomorphic
if and only if $\bar\partial g=0$ there (Cauchy--Riemann). In particular, since
$|w|^2=w\bar w$,
\begin{equation}\label{dbarpower}
\bar\partial|w|^s=\tfrac s2\,w\,|w|^{s-2},\qquad w\neq0,\ s\in\R.
\end{equation}

We will need a complex version of Green's formula on an annulus. Let
$0<\varepsilon<R$, let $D=\{\varepsilon\le|z|\le R\}$, and let
$u\in\Con^1(D)$. Then
\[
\int_{D}\bar\partial u\,\dint A=\frac{1}{2i}\oint_{\partial D} u \dint z
=\frac12\int_0^{2\pi}\Bigl(R\,e^{i\theta}u(Re^{i\theta})
-\varepsilon\,e^{i\theta}u(\varepsilon e^{i\theta})\Bigr)\,\dint\theta,
\]
where $\partial D$ is positively oriented: the outer circle counterclockwise
and the inner circle clockwise. In particular, if $u$ vanishes on
$\{|z|=R\}$, then
\begin{equation}\label{Greens}
\int_{D}\bar\partial u\,\dint A=-\frac{\varepsilon}2\int_0^{2\pi}e^{i\theta}\,u(\varepsilon e^{i\theta})\,\dint\theta .
\end{equation}

Finally, for $\Omega\in\Lup^1(\mathbb S^1)$ with Fourier coefficients
\[
\hat\Omega(k)=\frac{1}{2\pi}\int_0^{2\pi}\Omega(e^{i\theta})e^{-ik\theta}\dint\theta,
\]
the Fej\'er means
\[
\sigma_n\Omega(\zeta)=\sum_{k=-n}^n\Bigl(1-\frac{|k|}{n+1}\Bigr)\hat\Omega(k)\zeta^k,\qquad |\zeta|=1,
\]
are trigonometric polynomials, $\sigma_n\Omega\to\Omega$ in
$\Lup^1(\mathbb S^1)$ (see \cite{Gra14}), and $\sigma_n\Omega$ has mean zero
whenever $\Omega$ does, since $\widehat{\sigma_n\Omega}(0)=\hat\Omega(0)$.

\section{Proofs of Theorems~\ref{mainthm} and \ref{maincharthm}}\label{proofs}
In this section we prove our main results. Throughout, $f \in \Con_c^\infty(\R^d)$, so all of the operators below are well defined. Since every object involved commutes with translations, it suffices to prove the pointwise inequalities at $x = 0$.

\subsection{Two spherical functions}
For $\ep>0$, polar coordinates give
\[
\T_\Omega^\ep f(0)=\int_{|y|>\ep}\frac{\Omega(y')}{|y|^d}f(-y)\dy=\int_\Sp\Omega\,u_\ep\dint\mathcal H^{d-1},
\qquad
u_\ep(\theta)=\int_\ep^\infty \frac{f(-t\theta)}{t}\dint t .
\]
Similarly, for $s>0$,
\[
\fint_{B(0,s)}\Omega(y')f(-y)\dy=\int_\Sp\Omega\,w_s\dint\mathcal H^{d-1},
\qquad
w_s(\theta)=\frac{1}{v_ds^d}\int_0^s f(-t\theta)\,t^{d-1}\dint t ,
\]
where $v_d=|B(0,1)|=\omega_{d-1}/d$.

\begin{lemma}\label{lem:gradbounds}
For all $\ep,s>0$, the functions $u_\ep$ and $w_s$ belong to $\Con^\infty(\Sp)$, and
\[
\|\nabla_\Sp u_\ep\|_{\Lup^1(\Sp)}\leq \Int_1(|\nabla f|)(0)
\qquad\text{and}\qquad
\|\nabla_\Sp w_s\|_{\Lup^1(\Sp)}\leq \frac{2^{d-1}}{v_d}\,\M_1(|\nabla f|)(0).
\]
\end{lemma}

\begin{proof}
Extend $u_\ep$ to $\R^d\setminus\{0\}$ by $U(\xi)=\int_\ep^\infty f(-t\xi)\,t^{-1}\dint t$. Since $f\in\Con_c^\infty(\R^d)$ and $\ep>0$, we may differentiate under the integral sign, and $\nabla_\xi[f(-t\xi)]=-t\nabla f(-t\xi)$ gives
\[
\nabla_\Sp u_\ep(\theta)=\mathsf P_\theta\nabla U(\theta)=-\int_\ep^\infty\mathsf P_\theta\nabla f(-t\theta)\dint t .
\]
Since $|\mathsf P_\theta w|\leq|w|$, polar coordinates give
\begin{multline*}
\|\nabla_\Sp u_\ep\|_{\Lup^1(\Sp)}
\leq\int_\Sp\int_\ep^\infty|\nabla f(-t\theta)|\dint t\dint\mathcal H^{d-1}(\theta)
\\=\int_{|y|>\ep}\frac{|\nabla f(-y)|}{|y|^{d-1}}\dy
\leq\Int_1(|\nabla f|)(0).
\end{multline*}
The same argument applied to $w_s$ gives
\[
\nabla_\Sp w_s(\theta)=-\frac{1}{v_ds^d}\int_0^s\mathsf P_\theta\nabla f(-t\theta)\,t^{d}\dint t,
\]
and hence, since $|y|<s$ on the region of integration,
\[
\|\nabla_\Sp w_s\|_{\Lup^1(\Sp)}
\leq\frac{1}{v_ds^d}\int_{|y|<s}|y|\,|\nabla f(-y)|\dy
\leq\frac{1}{v_ds^{d-1}}\int_{B(0,s)}|\nabla f|\dy .
\]
If $Q$ is the cube of side length $2s$ centered at $0$, then $B(0,s)\subseteq Q$ and the right-hand side is at most $\frac{2^{d-1}}{v_d}\,\ell(Q)\fint_Q|\nabla f|\leq\frac{2^{d-1}}{v_d}\M_1(|\nabla f|)(0)$.
\end{proof}

\subsection{Proof of Theorem~\ref{mainthm}}
\begin{proof}[Proof of Theorem~\ref{mainthm}]
\emph{The upper bounds.} Let $\ep>0$. Since $\Omega$ has mean zero, we may subtract $(u_\ep)_\Sp$ from $u_\ep$. We then apply the Lorentz--H\"older inequality \eqref{lorentzholder} with $p=d-1$ and $q=\infty$, so that $p'=\frac{d-1}{d-2}$ and $q'=1$. Lemma~\ref{lem:lorentzpoincare} and Lemma~\ref{lem:gradbounds} then give
\begin{align*}
|\T_\Omega^\ep f(0)|
&=\Big|\int_\Sp\Omega\,\big(u_\ep-(u_\ep)_\Sp\big)\dint\mathcal H^{d-1}\Big|\\
&\leq \|\Omega\|_{\Lup^{d-1,\infty}(\Sp)}\,\|u_\ep-(u_\ep)_\Sp\|_{\Lup^{\frac{d-1}{d-2},1}(\Sp)}\\
&\leq c_d\|\Omega\|_{\Lup^{d-1,\infty}(\Sp)}\,\|\nabla_\Sp u_\ep\|_{\Lup^1(\Sp)}\\
&\leq c_d\|\Omega\|_{\Lup^{d-1,\infty}(\Sp)}\,\Int_1(|\nabla f|)(0),
\end{align*}
The right-hand side is independent of $\ep$, so taking the supremum over $\ep>0$ gives $\T_\Omega^\star f(0)\leq c_d\|\Omega\|_{\Lup^{d-1,\infty}(\Sp)}\Int_1(|\nabla f|)(0)$, and in particular \eqref{newptwiserough}. The same argument with $w_s$ in place of $u_\ep$ gives
\[
\Big|\fint_{B(0,s)}\Omega(y')f(-y)\dy\Big|\leq c_d\|\Omega\|_{\Lup^{d-1,\infty}(\Sp)}\M_1(|\nabla f|)(0),
\]
and taking the supremum over $s>0$ gives \eqref{newptwisemax}.

\emph{Sharpness.} Let $1\leq r<d-1$ and fix $a$ with $1<a<\frac{d-1}{r}$. In what follows $|\theta-\xi|$ is the Euclidean distance between $\theta,\xi\in\Sp$, $\dint\theta=\dint\mathcal H^{d-1}(\theta)$, and $c,C$ are positive constants depending only on $d$, $r$, and $a$.

Let $e_d=(0,\dots,0,1)$ and $\Omega_0(\theta)=|\theta-e_d|^{-a}$. Since $$\mathcal H^{d-1}(\{\theta\in\Sp:|\theta-e_d|<t\})\simeq t^{d-1}$$ for $0<t\leq2$ and $ar<d-1$ we have,
\[
\int_\Sp\Omega_0^r\,\dint\theta\leq C\int_0^2t^{-ar}\,t^{d-2}\dint t<\infty .
\]
Let $A=\{\theta\in\Sp:|\theta+e_d|<1\}$ and set
\[
\Omega=\Omega_0-\kappa\,\mathbf 1_A,\qquad \kappa=\frac{1}{\mathcal H^{d-1}(A)}\int_\Sp\Omega_0\,\dint\theta .
\]
Then $\Omega\in\Lup^r(\Sp)$ has mean zero. Moreover, $\Omega=\Omega_0$ on $\{\theta\in\Sp:|\theta-e_d|<1\}$, since every such $\theta$ lies outside $A$, as
$|\theta+e_d|^2=4-|\theta-e_d|^2>3.$

Fix $\psi\in\Con^\infty([0,\infty))$ with $0\leq\psi\leq1$, $\psi=1$ on $[0,1]$, $\psi=0$ on $[2,\infty)$, and $|\psi'|\leq2$, and for $0<\delta<\frac12$ set $\vp_\delta(\theta)=\psi(|\theta-e_d|/\delta)$. Then $\vp_\delta\in\Con^\infty(\Sp)$ is supported in $\{\theta\in \Sp:|\theta-e_d|<2\delta\}$ and $|\nabla_\Sp\vp_\delta|\leq2/\delta$, so
\begin{equation}\label{capest}
\int_\Sp\vp_\delta\,\dint\theta+\int_\Sp|\nabla_\Sp\vp_\delta|\,\dint\theta\leq c\delta^{d-1}+ C\delta^{d-2}\leq C\delta^{d-2}.
\end{equation}
On the other hand, $\Omega=\Omega_0\geq0$ on the support of $\vp_\delta$, and $\Omega_0\vp_\delta\geq\delta^{-a}$ on $\{\theta\in \Sp:|\theta-e_d|<\delta\}$, so
\begin{equation}\label{caplower}
\int_\Sp\Omega\,\vp_\delta\,\dint\theta\geq\int_{\{|\theta-e_d|<\delta\}}\Omega_0\,\dint\theta\geq c\,\delta^{d-1-a}.
\end{equation}

Fix $\eta\in\Con_c^\infty((\frac12,3))$ with $0\leq\eta\leq1$ and $\eta=1$ on $[1,2]$, and set $f_\delta(y)=\vp_\delta(-y')\,\eta(|y|)$. Then $f_\delta\in\Con_c^\infty(\R^d)$, since it is supported in the annulus $\frac12<|y|<3$, where it is smooth. Since $f_\delta(-y)=\vp_\delta(y')\eta(|y|)$ vanishes near $0$, no principal value is needed, and polar coordinates together with \eqref{caplower} give
\[
\T_\Omega f_\delta(0)=\Big(\int_\Sp\Omega\,\vp_\delta\,\dint\theta\Big)\int_0^\infty\eta(\rho)\,\frac{\dint\rho}{\rho}\geq c\,\delta^{d-1-a},
\]
and
\begin{multline*}
\M_\Omega f_\delta(0)\geq\fint_{B(0,3)}\Omega(y')f_\delta(-y)\dy\\=\frac{1}{|B(0,3)|}\Big(\int_\Sp\Omega\,\vp_\delta\,\dint\theta\Big)\int_0^3\eta(\rho)\rho^{d-1}\dint\rho\geq c\,\delta^{d-1-a}.
\end{multline*}

For the right-hand sides, \eqref{polargrad} and the fact that $\theta\mapsto-\theta$ is an isometry of $\Sp$ give
\[
|\nabla f_\delta(y)|\leq\frac{|\nabla_\Sp\vp_\delta(-y')|}{|y|}\,\eta(|y|)+\vp_\delta(-y')\,|\eta'(|y|)|.
\]
Integrating in polar coordinates and using \eqref{capest},
\begin{multline*}
\Int_1(|\nabla f_\delta|)(0)
\leq\int_\Sp|\nabla_\Sp\vp_\delta|\,\dint\theta\int_0^\infty\eta(\rho)\,\frac{\dint\rho}{\rho}
\\+\int_\Sp\vp_\delta\,\dint\theta\int_0^\infty|\eta'(\rho)|\dint\rho\leq C\delta^{d-2}.
\end{multline*}
Let $Q$ be a cube containing $0$. If $Q$ meets the support of $f_\delta$, then $\ell(Q)\geq c$, so similarly
\[
\M_1(|\nabla f_\delta|)(0)\leq C\int_{\R^d}|\nabla f_\delta|\dy\leq C\delta^{d-2}.
\]
Hence
\[
\frac{\T_\Omega f_\delta(0)}{\Int_1(|\nabla f_\delta|)(0)}\geq c\,\delta^{1-a}
\qquad\text{and}\qquad
\frac{\M_\Omega f_\delta(0)}{\M_1(|\nabla f_\delta|)(0)}\geq c\,\delta^{1-a},
\]
both of which tend to infinity as $\delta\to0^+$ because $a>1$.
\end{proof}

\subsection{Proof of Theorem~\ref{maincharthm}}

\begin{proof}[Proof of Theorem~\ref{maincharthm}]

Let us first note that each of (1)--(5) implies that $\Omega$ has mean zero. For (1) and (5), take $\vp\equiv1$ in \eqref{dualnorm} and \eqref{weakdiv}. For (4), take $E=\Sp$, whose boundary is empty. For (2) and (3), suppose that $\int_{\Sp}\Omega \neq 0$. To see that (3) fails, let $f \in \Con_c^\infty(\R^d)$ with $f = 1$ on the unit ball. For $0 < \ep < 1$, polar coordinates give
\[
  \T^\ep_\Omega f(0)
  = \int_{\ep<|y|<1} \frac{\Omega(y')}{|y|^d}\dy + \int_{|y|\ge 1} \frac{\Omega(y')}{|y|^d}\, f(-y)\dy
  = \log\Bigl(\frac{1}{\ep}\Bigr)\int_{\Sp}\Omega + O(1).
\]
Hence $|\T^\ep_\Omega f(0)| \to \infty$ as $\ep \to 0$, so the principal value defining $\T_\Omega f(0)$ does not exist and (3) cannot hold.

The quantity $\M_\Omega f$ is finite for any $\Omega\in \Lup^1(\Sp)$ and $f\in \Con_c^\infty(\R^d)$, but (2) still fails. Fix $\psi \in \Con^\infty(\R)$ with $0 \le \psi \le 1$, $\psi = 1$ on $(-\infty, 0]$, and $\psi = 0$ on $[1, \infty)$. For $L \ge 1$, set
\[
  f_L(y) = \psi\Bigl(\frac{\log|y|}{L}\Bigr), \qquad f_L(0) = 1.
\]
Then $f_L = 1$ on the unit ball, $f_L = 0$ for $|y| \ge e^L$, and $f_L \in \Con_c^\infty(\R^d)$. Its gradient is supported in $1 \le |y| \le e^L$, where
\[
  |\nabla f_L(y)| \le \frac{\|\psi'\|_{\Lup^\infty(\R)}}{L|y|}.
\]
On the one hand, for $r \le 1$,
\[
  \M_\Omega f_L(0) \ge \Bigl|\fint_{B(0,r)}\Omega(y')f_L(-y)\dy\Bigr| = \Bigl|\fint_{B(0,r)}\Omega(y')\dy\Bigr| = \frac{1}{\omega_{d-1}}\Bigl|\int_{\Sp}\Omega\Bigr|.
\]
On the other hand, every cube $Q\ni 0$ is contained in $B(0,r)$ with $r=\sqrt d\,\ell(Q)$, so $\ell(Q)\fint_Q|\nabla f_L|\leq c_d\, r^{1-d}\int_{|y|<r}|\nabla f_L(y)|\dy$. The last integral vanishes for $r \le 1$, while for $r > 1$,
\[
  r^{1-d}\int_{|y|<r}|\nabla f_L(y)|\dy
  \le \frac{c_d\|\psi'\|_{\Lup^\infty(\R)}}{L}\, r^{1-d}\int_1^r s^{d-2}\,\dint s
  \le \frac{c_d\|\psi'\|_{\Lup^\infty(\R)}}{L}.
\]
Thus $\M_1(|\nabla f_L|)(0) \leq c_d\|\psi'\|_{\Lup^\infty(\R)}/L$, while $\M_\Omega f_L(0) \geq \frac{1}{\omega_{d-1}}|\int_{\Sp}\Omega|$. Letting $L \to \infty$ shows that (2) cannot hold.

Consequently, whenever one of (1)--(5) holds, $\Omega$ has mean zero; we use this freely in the remaining steps.

Write $\|\Omega\|=\|\Omega\|_{\dot\W^{-1,\infty}(\Sp)}$. By the definition \eqref{dualnorm}, and since $\nabla_\Sp\vp=0$ forces $\vp$ to be constant, (1) is equivalent to
\begin{equation}\label{dualineq}
\Big|\int_\Sp\Omega\,\vp\dint\mathcal H^{d-1}\Big|\leq\|\Omega\|\,\|\nabla_\Sp\vp\|_{\Lup^1(\Sp)}\qquad\text{for all }\vp\in\Con^\infty(\Sp).
\end{equation}
We show that each of (2)--(5) is equivalent to (1). By translation invariance it suffices to prove (2) and (3) at $x=0$.

\emph{(1) $\Rightarrow$ (2), (3).} Let $\ep,s>0$, and let $u_\ep$, $w_s$ be as in Lemma~\ref{lem:gradbounds}. By \eqref{dualineq} and Lemma~\ref{lem:gradbounds},
\[
|\T_\Omega^\ep f(0)|=\Big|\int_\Sp\Omega\,u_\ep\dint\mathcal H^{d-1}\Big|\leq\|\Omega\|\,\|\nabla_\Sp u_\ep\|_{\Lup^1(\Sp)}\leq\|\Omega\|\,\Int_1(|\nabla f|)(0),
\]
and
\[
\Big|\fint_{B(0,s)}\Omega(y')f(-y)\dy\Big|=\Big|\int_\Sp\Omega\,w_s\dint\mathcal H^{d-1}\Big|\leq\frac{2^{d-1}}{v_d}\|\Omega\|\,\M_1(|\nabla f|)(0).
\]
Taking suprema over $\ep>0$ and $s>0$ gives (3), in fact for $\T_\Omega^\star$, and (2).

\emph{(2) $\Rightarrow$ (1) and (3) $\Rightarrow$ (1).} Let $\vp\in\Con^\infty(\Sp)$ with $\|\nabla_\Sp \vp\|_{\Lup^1(\Sp)}\leq 1$. By the first step, $\Omega$ has mean zero, so replacing $\vp$ by $\vp-\vp_\Sp$ changes neither side of \eqref{dualineq}, and we may assume $\vp_\Sp=0$. Fix $\eta\in\Con_c^\infty((\frac12,3))$ with $0\leq\eta\leq1$ and $\eta=1$ on $[1,2]$, and set $f(y)=\vp(-y')\,\eta(|y|)$. Then $f\in\Con_c^\infty(\R^d)$ is supported in the annulus $\{y:\frac12<|y|<3\}$, and
\[
\T_\Omega f(0)=\int_{\R^d}\frac{\Omega(y')}{|y|^d}f(-y)\dy=\Big(\int_\Sp\Omega\,\vp\dint\mathcal H^{d-1}\Big)\int_0^\infty\eta(\rho)\,\frac{\dint\rho}{\rho},
\]
\[
\fint_{B(0,3)}\Omega(y')f(-y)\dy=\frac{1}{v_d3^d}\Big(\int_\Sp\Omega\,\vp\dint\mathcal H^{d-1}\Big)\int_0^3\eta(\rho)\rho^{d-1}\dint\rho .
\]
Both radial integrals are bounded below by a constant depending only on $d$, so
\begin{equation}\label{charlower}
\Big|\int_\Sp\Omega\,\vp\dint\mathcal H^{d-1}\Big|\leq C_d\min\big\{|\T_\Omega f(0)|,\ \M_\Omega f(0)\big\}.
\end{equation}
On the other hand, by \eqref{polargrad},
\[
|\nabla f(y)|\leq\frac{|\nabla_\Sp\vp(-y')|}{|y|}\,\eta(|y|)+|\vp(-y')|\,|\eta'(|y|)|,
\]
so integrating in polar coordinates and using \eqref{poincareL1} (recall that $\vp_\Sp=0$),
\[
\Int_1(|\nabla f|)(0)\leq C_d\big(\|\nabla_\Sp\vp\|_{\Lup^1(\Sp)}+\|\vp\|_{\Lup^1(\Sp)}\big)\leq C_d\|\nabla_\Sp\vp\|_{\Lup^1(\Sp)}\leq C_d.
\]
Let $Q\subseteq\R^d$ be a cube with $0\in Q$. If $Q$ meets the support of $f$, then $\ell(Q)\geq(2\sqrt d)^{-1}$, so $\ell(Q)^{1-d}\leq(2\sqrt d)^{d-1}$ and similarly
\[
\M_1(|\nabla f|)(0)\leq C_d\int_{\R^d}|\nabla f|\dy\leq C_d\big(\|\nabla_\Sp\vp\|_{\Lup^1(\Sp)}+\|\vp\|_{\Lup^1(\Sp)}\big)\leq C_d.
\]
Combining these with \eqref{charlower}, if (2) or (3) holds with constant $c$, then \eqref{dualineq} holds with $\|\Omega\|$ replaced by $C_dc$, so $\|\Omega\|\leq C_dc$.

\emph{(5) $\Rightarrow$ (1).} If $\Omega=\Div_\Sp \mathbf{V}$, then by \eqref{weakdiv}, for every $\vp\in\Con^\infty(\Sp)$,
\[
\Big|\int_\Sp\Omega\,\vp\dint\mathcal H^{d-1}\Big|=\Big|\int_\Sp \mathbf{V}\cdot\nabla_\Sp\vp\dint\mathcal H^{d-1}\Big|\leq\| \mathbf{V}\|_{\Lup^\infty(\Sp,\R^d)}\|\nabla_\Sp\vp\|_{\Lup^1(\Sp)},
\]
so $\|\Omega\|\leq\| \mathbf{V}\|_{\Lup^\infty(\Sp,\R^d)}$.

\emph{(1) $\Rightarrow$ (5).} Let $\mathcal G=\{\nabla_\Sp\vp:\vp\in\Con^\infty(\Sp)\}$, which is a subspace of $\Lup^1(\Sp,\R^d)$, and define $\ell:\mathcal G\to\R$ by
\[
\ell(\nabla_\Sp\vp)=-\int_\Sp\Omega\,\vp\dint\mathcal H^{d-1}.
\]
This is well defined: if $\nabla_\Sp\vp_1=\nabla_\Sp\vp_2$, then $\vp_1-\vp_2$ is constant on $\Sp$, and since $\Omega$ has mean zero by the first step, $\ell(\nabla_\Sp\vp_1)=\ell(\nabla_\Sp\vp_2)$. Moreover, $|\ell(\mathbf G)|\leq\|\Omega\|\,\|\mathbf G\|_{\Lup^1(\Sp,\R^d)}$ for $\mathbf G\in \mathcal G$ by \eqref{dualineq}. By the Hahn--Banach theorem, $\ell$ extends to a bounded linear functional on $\Lup^1(\Sp,\R^d)$ with norm at most $\|\Omega\|$. Since $\Lup^1(\Sp,\R^d)^*=\Lup^\infty(\Sp,\R^d)$, there is $ \mathbf{V}\in\Lup^\infty(\Sp,\R^d)$ with $\| \mathbf{V}\|_{\Lup^\infty(\Sp,\R^d)}\leq\|\Omega\|$ and
\[
-\int_\Sp\Omega\,\vp\dint\mathcal H^{d-1}=\int_\Sp \mathbf{V}\cdot\nabla_\Sp\vp\dint\mathcal H^{d-1}\qquad\text{for all }\vp\in\Con^\infty(\Sp),
\]
which shows that $\Omega=\Div_\Sp \mathbf{V}$. Combined with the previous step, we have $\|\Omega\|=\| \mathbf{V}\|_{\Lup^\infty(\Sp,\R^d)}$, so the minimum in Theorem~\ref{maincharthm} is attained and equals $\|\Omega\|$.

\emph{(4) $\Rightarrow$ (1).} Let $\vp\in\Con^\infty(\Sp)$ with $m=\min\vp$ and $M=\max\vp$. Since $\vp(\theta)=m+\int_m^M\mathbf 1_{\{\vp>t\}}(\theta)\dint t$ and $\Omega$ has mean zero by the first step, the layer-cake formula gives
\[
\int_\Sp\Omega\,\vp\dint\mathcal H^{d-1}=\int_m^M\Big(\int_{\{\vp>t\}}\Omega\dint\mathcal H^{d-1}\Big)\dint t .
\]
By Sard's theorem, almost every $t\in(m,M)$ is a regular value of $\vp$, and for such $t$ the set $\{\vp>t\}$ is open with smooth boundary $\{\vp=t\}$. Hence, by (4) and the coarea formula on $\Sp$,
\[
\Big|\int_\Sp\Omega\,\vp\dint\mathcal H^{d-1}\Big|\leq C\int_m^M\mathcal H^{d-2}(\{\vp=t\})\dint t=C\,\|\nabla_\Sp\vp\|_{\Lup^1(\Sp)},
\]
so $\|\Omega\|\leq C$.

\emph{(1) $\Rightarrow$ (4).} Let $E\subseteq\Sp$ be open with smooth
boundary, let $h$ be the signed geodesic distance to $\partial E$ given in \eqref{geodes} and let $t_0>0$ be as in
Lemma~\ref{lem:signed-distance}. Thus $h$ is smooth with
$|\nabla_\Sp h|=1$ on $\{|h|<t_0\}$, and
$\mathcal H^{d-2}(\{h=t\})\to\mathcal H^{d-2}(\partial E)$ as $t\to0$. Fix
a nondecreasing $\chi\in\Con^\infty(\R)$ with $\chi=0$ on $(-\infty,0]$
and $\chi=1$ on $[1,\infty)$, and for $k>1/t_0$ set $\vp_k=\chi(kh)$.
Then $\vp_k\in\Con^\infty(\Sp)$, $0\leq\vp_k\leq1$, and
$\vp_k\to\mathbf 1_E$ pointwise. Since $\nabla_\Sp\vp_k$ vanishes outside
$\{0\le h\le 1/k\}\subseteq\{|h|<t_0\}$, the coarea formula
\eqref{coarea} with $U=\{|h|<t_0\}$, $u=h$ and $g=k\chi'(kh)$ gives
\[
\|\nabla_\Sp\vp_k\|_{\Lup^1(\Sp)}=\int_0^{1/k}k\chi'(kt)\,\mathcal H^{d-2}(\{h=t\})\dint t\longrightarrow\mathcal H^{d-2}(\partial E),
\]
since $\int_0^{1/k}k\chi'(kt)\dint t=1$. Since
$|\Omega\vp_k|\leq|\Omega|\in\Lup^1(\Sp)$, the dominated convergence
theorem and \eqref{dualineq} give
\[
\Big|\int_E\Omega\dint\mathcal H^{d-1}\Big|=\lim_{k\to\infty}\Big|\int_\Sp\Omega\,\vp_k\dint\mathcal H^{d-1}\Big|\leq\|\Omega\|\,\mathcal H^{d-2}(\partial E).
\]
Together with the previous step, this shows that the best constant in
(4) equals $\|\Omega\|$.
\end{proof}

\section{Operators in the plane}\label{plane}

In dimension $d=2$, we use complex analysis. Let $\Omega\in\Lup^1(\mathbb S^1)$ be complex valued and have mean zero. For $f\in\Con_c^\infty(\C)$, we write
$$\T_\Omega f(z)=\textup{p.v.}\int_{\C}\frac{\Omega(w/|w|)}{|w|^2}f(z-w)\dint A(w)=\lim_{\ep\ra 0}\int_{|w|>\ep}\frac{\Omega(w/|w|)}{|w|^2}f(z-w)\dint A(w).$$
The key observation in the proof of Theorem~\ref{mainthmplane} is that $\T_\Omega$ can be written as a twisted Cauchy transform,
$$\T_\Omega f(z)=\int_{\C}G(w)\,\bar{\partial}f(z-w)\dint A(w),$$
where $|G(w)|\leq C|w|^{-1}$ and $\bar\partial G(w)=\Omega(w/|w|)/|w|^2$ for $w\neq0$. To construct $G$ we need a bounded primitive of $\Omega$.

\begin{lemma}\label{primitiveV}
Let $\Omega\in\Lup^1(\mathbb S^1)$ have mean zero, and let $V_0$ be the primitive of $\Omega$, that is,
$$V_0(e^{i\theta})=\int_0^\theta\Omega(e^{it})\dint t .$$
If we define
\begin{equation}\label{defV}
V=V_0-\frac{1}{2\pi}\int_0^{2\pi}V_0(e^{it})\dint t,
\end{equation}
then $V$ is continuous on $\mathbb S^1$, has mean zero, and satisfies
$$\|V\|_{\Lup^\infty(\mathbb S^1)}\leq\frac12\|\Omega\|_{\Lup^1(\mathbb S^1)}.$$
\end{lemma}

\begin{proof}
Since $\Omega$ has mean zero, the function $\theta\mapsto\int_0^\theta\Omega(e^{it})\dint t$ is $2\pi$-periodic, so $V_0$, and hence $V$, is a well-defined continuous function on $\mathbb S^1$. By construction, $V$ has mean zero. Given $0\le\theta<\psi\le2\pi$,
\[
V(e^{i\psi})-V(e^{i\theta})=\int_\theta^\psi\Omega(e^{it})\dint t .
\]
Since $\int_{\mathbb S^1}\Omega=0$, we also have
\[
V(e^{i\psi})-V(e^{i\theta})=-\int_{[0,2\pi]\backslash(\theta,\psi)}\Omega(e^{it})\dint t .
\]
Set $a=\int_\theta^\psi|\Omega(e^{it})|\dint t$ and
$b=\int_{[0,2\pi]\backslash(\theta,\psi)}|\Omega(e^{it})|\dint t$, so that
$a+b=\|\Omega\|_{\Lup^1(\mathbb S^1)}$. The two identities give
\[
|V(e^{i\psi})-V(e^{i\theta})|\le\min\{a,b\}\le\frac{a+b}{2}
=\frac12\|\Omega\|_{\Lup^1(\mathbb S^1)} .
\]
By symmetry, this holds for all $\theta,\psi\in[0,2\pi]$. Finally, since $V$ has
mean zero,
\[
|V(e^{i\theta})|
=\Bigl|\frac1{2\pi}\int_0^{2\pi}\bigl(V(e^{i\theta})-V(e^{i\psi})\bigr)\dint\psi\Bigr|
\le\frac12\|\Omega\|_{\Lup^1(\mathbb S^1)},\qquad 0\le\theta\le2\pi. \qedhere
\]
\end{proof}

\begin{lemma}\label{dbarG}
Let $\Omega(\zeta)$, $|\zeta|=1$, be a trigonometric polynomial with mean zero, let $V$ be defined by \eqref{defV}, and set
$$G(w)=-2i\,\frac{V(w/|w|)}{w},\qquad w\in\C\backslash\{0\}.$$
Then $G\in\Con^\infty(\C\backslash\{0\})$ and
$$\bar{\partial}G(w)=\frac{\Omega(w/|w|)}{|w|^2},\qquad w\in\C\backslash\{0\}.$$
\end{lemma}

\begin{proof}
Since $V$ and $G$ depend linearly on $\Omega$, it suffices to consider $\Omega(\zeta)=\zeta^k$ for $k\in\Z\backslash\{0\}$. Then
$$V(\zeta)=\frac{\zeta^k}{ik},\qquad|\zeta|=1,$$
since $\frac{d}{d\theta}\frac{e^{ik\theta}}{ik}=e^{ik\theta}$ and $\zeta^k$ has mean zero. For $w\in\C\backslash\{0\}$ we have
$$G(w)=-2i\frac{V(w/|w|)}{w}=-\frac{2}{k}w^{k-1}|w|^{-k},$$
which is smooth on $\C\backslash\{0\}$. Since $w^{k-1}$ is holomorphic there, \eqref{dbarpower} gives
$$\bar{\partial}G(w)=-\frac2kw^{k-1}\Big(-\frac{k}{2}\Big)w|w|^{-k-2}=w^k|w|^{-k-2}=\frac{(w/|w|)^k}{|w|^{2}}=\frac{\Omega(w/|w|)}{|w|^2}. \qedhere$$
\end{proof}

\begin{lemma}\label{truncid}
Let $\Omega\in\Lup^1(\mathbb S^1)$ have mean zero, and let $V$ be defined by \eqref{defV}. For every $\ep>0$, $f\in\Con_c^\infty(\C)$, and $z\in\C$ we have
\begin{multline}\label{greensoper}
\int_{|w|>\ep}\frac{\Omega(w/|w|)}{|w|^2}f(z-w)\dint A(w)=-2i\int_{|w|>\ep}\frac{V(w/|w|)}{w}\,\bar{\partial}f(z-w)\dint A(w)\\
+i\int_0^{2\pi}V(e^{i\theta})f(z-\ep e^{i\theta})\dint\theta.
\end{multline}
\end{lemma}

\begin{proof}
Replacing $f$ by $f(\cdot+z)$, it suffices to prove \eqref{greensoper} for $z=0$. Fix $f\in\Con_c^\infty(\C)$ and $\ep>0$, and choose $R>\ep$ so that $f(-w)=0$ for $|w|\geq R$.

\emph{Trigonometric polynomials.} Suppose first that $\Omega$ is a trigonometric polynomial with mean zero, and let $G$ be as in Lemma~\ref{dbarG}. Then
$$\int_{|w|>\ep}\frac{\Omega(w/|w|)}{|w|^2}f(-w)\dint A(w)=\int_{\ep<|w|<R}\bar{\partial}G(w)f(-w)\dint A(w).$$
Since $\bar{\partial}\big(G(w)f(-w)\big)=(\bar{\partial}G)(w)f(-w)-G(w)(\bar{\partial}f)(-w)$, we have
\begin{multline*}
\int_{|w|>\ep}\frac{\Omega(w/|w|)}{|w|^2}f(-w)\dint A(w)=\int_{\ep<|w|<R}\bar{\partial}\big(G(w)f(-w)\big)\dint A(w)\\
+\int_{|w|>\ep}G(w)(\bar{\partial}f)(-w)\dint A(w).
\end{multline*}
The function $u(w)=G(w)f(-w)$ is $\Con^1$ on $\{\ep\le|w|\le R\}$ and vanishes on $\{|w|=R\}$. Since $e^{i\theta}G(\ep e^{i\theta})=-2iV(e^{i\theta})/\ep$, Green's formula \eqref{Greens} gives
\begin{multline*}
\int_{\ep<|w|<R}\bar{\partial}\big(G(w)f(-w)\big)\dint A(w)=-\frac{\ep}{2}\int_0^{2\pi}e^{i\theta}G(\ep e^{i\theta})f(-\ep e^{i\theta})\dint\theta\\
=i\int_0^{2\pi}V(e^{i\theta})f(-\ep e^{i\theta})\dint\theta.
\end{multline*}
This shows that \eqref{greensoper} holds when $\Omega$ is a trigonometric polynomial.

\emph{General $\Omega$.} For general $\Omega\in\Lup^1(\mathbb S^1)$ with mean zero, let $\Omega_n=\sigma_n\Omega$ be the Fej\'er means of $\Omega$, so that each $\Omega_n$ is a trigonometric polynomial with mean zero and $\Omega_n\ra\Omega$ in $\Lup^1(\mathbb S^1)$. Let $V_n$ be the primitive of $\Omega_n$ defined by \eqref{defV}. Then $V_n-V$ is the primitive of $\Omega_n-\Omega$ also given by \eqref{defV}, and hence, by Lemma~\ref{primitiveV},
$$\|V_n-V\|_{\Lup^\infty(\mathbb S^1)}\leq\frac12\|\Omega_n-\Omega\|_{\Lup^1(\mathbb S^1)}\ra 0,\qquad n\ra\infty.$$
Since \eqref{greensoper} holds for each $\Omega_n$, it suffices to show that each term in \eqref{greensoper} for $\Omega_n$ converges to the corresponding term for $\Omega$. For the first term,
\begin{align*}
\lefteqn{\left|\int_{|w|>\ep}\frac{\Omega_n(w/|w|)}{|w|^2}f(-w)\dint A(w)-\int_{|w|>\ep}\frac{\Omega(w/|w|)}{|w|^2}f(-w)\dint A(w)\right|}\\
&\qquad\qquad\leq\int_{\ep<|w|<R}\frac{|\Omega_n(w/|w|)-\Omega(w/|w|)|}{|w|^2}|f(-w)|\dint A(w)\\
&\qquad\qquad\leq\|\Omega_n-\Omega\|_{\Lup^1(\mathbb S^1)}\|f\|_{\Lup^\infty(\C)}\int_\ep^R\frac{1}{\rho}\dint\rho\\
&\qquad\qquad=\|\Omega_n-\Omega\|_{\Lup^1(\mathbb S^1)}\|f\|_{\Lup^\infty(\C)}\log(R/\ep)\ra0,\qquad n\ra\infty.
\end{align*}
For the second term,
\begin{align*}
\lefteqn{\left|\int_{|w|>\ep}\frac{V_n(w/|w|)}{w}\bar{\partial}f(-w)\dint A(w)-\int_{|w|>\ep}\frac{V(w/|w|)}{w}\bar{\partial}f(-w)\dint A(w)\right|}\\
&\qquad\qquad\leq\int_{\ep<|w|<R}\frac{|V_n(w/|w|)-V(w/|w|)|}{|w|}|\bar{\partial}f(-w)|\dint A(w)\\
&\qquad\qquad\leq\|V_n-V\|_{\Lup^\infty(\mathbb S^1)}\|\bar{\partial}f\|_{\Lup^\infty(\C)}\int_{\ep<|w|<R}\frac{\dint A(w)}{|w|}\\
&\qquad\qquad=2\pi(R-\ep)\|V_n-V\|_{\Lup^\infty(\mathbb S^1)}\|\bar{\partial}f\|_{\Lup^\infty(\C)}\ra0,\qquad n\ra\infty.
\end{align*}
Finally, for the last term,
\begin{align*}
\lefteqn{\left|\int_0^{2\pi}V_n(e^{i\theta})f(-\ep e^{i\theta})\dint\theta-\int_0^{2\pi}V(e^{i\theta})f(-\ep e^{i\theta})\dint\theta\right|}\\
&\qquad\qquad\leq\int_0^{2\pi}|V_n(e^{i\theta})-V(e^{i\theta})||f(-\ep e^{i\theta})|\dint\theta\\
&\qquad\qquad\leq2\pi\|V_n-V\|_{\Lup^\infty(\mathbb S^1)}\|f\|_{\Lup^\infty(\C)}\ra0,\qquad n\ra\infty.
\end{align*}
It follows that \eqref{greensoper} holds for general $\Omega \in \Lup^1(\mathbb S^1)$.
\end{proof}

\begin{proposition}\label{twistCauchyprop}
Let $\Omega\in\Lup^1(\mathbb S^1)$ have mean zero, and let $V$ be defined by \eqref{defV}. For every $f\in\Con_c^\infty(\C)$ and $z\in\C$, the limit defining $\T_\Omega f(z)$ exists, and
\begin{equation}\label{twistCauchy}
\T_\Omega f(z)=-2i\int_{\C}\frac{V(w/|w|)}{w}\,\bar{\partial}f(z-w)\dint A(w)=2i\int_{\C}\frac{V(w/|w|)}{\bar w}\,\partial f(z-w)\dint A(w).
\end{equation}
Consequently, the estimate \eqref{newptwiseroughplane} holds, and Theorem~\ref{mainthmplane} follows.
\end{proposition}

\begin{proof}
We let $\ep\ra0$ in \eqref{greensoper}. Since $|V(w/|w|)/w|\le\|V\|_{\Lup^\infty(\mathbb S^1)}|w|^{-1}$ and $\bar\partial f$ is bounded with compact support, the first integral in \eqref{twistCauchy} converges absolutely, and hence the first term on the right-hand side of \eqref{greensoper} converges to the first integral in \eqref{twistCauchy} as $\ep\ra0$. For the second term, since $V$ has mean zero on $\mathbb S^1$,
\begin{multline*}
\left|\int_0^{2\pi}V(e^{i\theta})f(z-\ep e^{i\theta})\dint\theta\right|=\left|\int_0^{2\pi}V(e^{i\theta})\big[f(z-\ep e^{i\theta})-f(z)\big]\dint\theta\right|\\
\leq2\pi\ep\,\|V\|_{\Lup^\infty(\mathbb S^1)}\|Df\|_{\Lup^\infty(\C)}\ra0,\qquad\ep\ra0.
\end{multline*}
Therefore the limit defining $\T_\Omega f(z)$ exists and equals the first integral in \eqref{twistCauchy}.

For the second identity, note that $\bar\Omega$ has mean zero, its primitive defined by \eqref{defV} is $\bar V$, and $\overline{\T_\Omega f}=\T_{\bar\Omega}\bar f$; applying the first identity to $\bar\Omega$ and $\bar f$, using $\bar\partial\bar f=\overline{\partial f}$, and taking complex conjugates gives the second identity. Finally, by \eqref{twistCauchy} and Lemma~\ref{primitiveV},
\[
|\T_\Omega f(z)|\leq 2\|V\|_{\Lup^\infty(\mathbb S^1)}\int_{\C}\frac{|\bar{\partial}f(z-w)|}{|w|}\dint A(w)\leq\|\Omega\|_{\Lup^1(\mathbb S^1)}\Int_1(|\bar\partial f|)(z),
\]
and in the same way $|\T_\Omega f(z)|\leq\|\Omega\|_{\Lup^1(\mathbb S^1)}\Int_1(|\partial f|)(z)$. This proves \eqref{newptwiseroughplane}.
\end{proof}

\begin{proof}[Proof of Corollary~\ref{boundedplane}]
In the plane, the Hardy--Littlewood--Sobolev inequality states that $\Int_1:\Lup^p(\C)\ra\Lup^{\frac{2p}{2-p}}(\C)$ is bounded for $1<p<2$. Combined with \eqref{newptwiseroughplane} and $|\bar\partial f|\leq|Df|$, this gives
\[
\|\T_\Omega f\|_{\Lup^{\frac{2p}{2-p}}(\C)}\leq C_p\|\Omega\|_{\Lup^1(\mathbb S^1)}\|\bar\partial f\|_{\Lup^p(\C)}\leq C_p\|\Omega\|_{\Lup^1(\mathbb S^1)}\|Df\|_{\Lup^p(\C)},\qquad f\in\Con_c^\infty(\C).
\]
Since $\Con_c^\infty(\C)$ is dense in $\dot\W^{1,p}(\C)$, $\T_\Omega$ extends uniquely to a bounded operator $\dot\W^{1,p}(\C)\ra\Lup^{\frac{2p}{2-p}}(\C)$.
\end{proof}

For the endpoint results we record some background on the spaces. Let $\dot\W^{1,1}_{\bar\partial}(\C)$ and $\dot\W^{1,1}_{\partial}(\C)$ be the completions of $\Con_c^\infty(\C)$ with respect to $\|\bar\partial f\|_{\Lup^1(\C)}$ and $\|\partial f\|_{\Lup^1(\C)}$. These are norms on $\Con_c^\infty(\C)$: if $\bar\partial f\equiv0$ (respectively $\partial f\equiv0$), then $f$ (respectively $\bar f$) is entire with compact support, hence $f\equiv0$. By the Cauchy--Pompeiu formula and the weak-type Hardy--Littlewood--Sobolev inequality,
\begin{equation}\label{onesidedemb}
\|f\|_{\Lup^{2,\infty}(\C)}\leq C\min\big\{\|\bar\partial f\|_{\Lup^1(\C)},\|\partial f\|_{\Lup^1(\C)}\big\},\qquad f\in\Con_c^\infty(\C),
\end{equation}
so both spaces embed continuously into $\Lup^{2,\infty}(\C)$, and their elements are functions whose distributional derivative $\bar\partial f$, respectively $\partial f$, belongs to $\Lup^1(\C)$. In particular, their sum is a Banach space with the norm
\[
\|f\|=\inf\Big\{\|\bar\partial g\|_{\Lup^1(\C)}+\|\partial h\|_{\Lup^1(\C)}:\ f=g+h,\ g\in\dot\W^{1,1}_{\bar\partial}(\C),\ h\in\dot\W^{1,1}_{\partial}(\C)\Big\}.
\]
Since $\|Df\|_{\Lup^1(\C)}=\|\bar\partial f\|_{\Lup^1(\C)}+\|\partial f\|_{\Lup^1(\C)}$, a standard mollification and cutoff argument gives $\dot\W^{1,1}(\C)=\dot\W^{1,1}_{\bar\partial}(\C)\cap\dot\W^{1,1}_{\partial}(\C)$. Finally, since $\bar\partial\bar f=\overline{\partial f}$, complex conjugation maps $\dot\W^{1,1}_{\bar\partial}(\C)$ isometrically onto $\dot\W^{1,1}_{\partial}(\C)$.

Neither of $\dot\W^{1,1}_{\bar\partial}(\C)$ and $\dot\W^{1,1}_{\partial}(\C)$ contains the other. For $0<\ep<\frac18$, let $f_\ep(z)=\chi_\ep(|z|)/z$, where $\chi_\ep\in\Con_c^\infty((\ep,\frac12))$ satisfies $0\leq\chi_\ep\leq1$, $\chi_\ep=1$ on $[2\ep,\frac14]$, and $\int_0^\infty|\chi_\ep'|=2$. Since $1/z$ is holomorphic away from the origin,
\[
\|\bar\partial f_\ep\|_{\Lup^1(\C)}=\pi\int_0^\infty|\chi_\ep'(r)|\dint r=2\pi,
\qquad
\|\partial f_\ep\|_{\Lup^1(\C)}\geq\int_{2\ep<|z|<\frac14}\frac{\dint A(z)}{|z|^{2}}=2\pi\log\frac1{8\ep}.
\]
If $\dot\W^{1,1}_{\bar\partial}(\C)\subseteq\dot\W^{1,1}_{\partial}(\C)$, the inclusion would have closed graph, since both spaces embed continuously in $\Lup^{2,\infty}(\C)$. By the closed graph theorem it would be bounded, giving $\|\partial f\|_{\Lup^1(\C)}\leq C\|\bar\partial f\|_{\Lup^1(\C)}$ for $f\in\Con_c^\infty(\C)$, which the functions $f_\ep$ rule out. By conjugation, $\dot\W^{1,1}_{\partial}(\C)\not\subseteq\dot\W^{1,1}_{\bar\partial}(\C)$ as well. Consequently,
\[
\dot\W^{1,1}(\C)\subsetneq\dot\W^{1,1}_{\bar\partial}(\C), \dot\W^{1,1}_{\partial}(\C) \subsetneq\dot\W^{1,1}_{\bar\partial}(\C)+\dot\W^{1,1}_{\partial}(\C).
\]

\begin{proof}[Proof of Corollary~\ref{boundedplaneweak}]
The endpoint Hardy--Littlewood--Sobolev inequality states that $\Int_1:\Lup^1(\C)\ra\Lup^{2,\infty}(\C)$ is bounded. Combined with \eqref{newptwiseroughplane}, this gives \eqref{weakSob}. Since $\Lup^{2,\infty}(\C)$ is complete, \eqref{weakSob} and density show that $\T_\Omega$ extends uniquely to bounded operators $\dot\W^{1,1}_{\bar\partial}(\C)\ra\Lup^{2,\infty}(\C)$ and $\dot\W^{1,1}_{\partial}(\C)\ra\Lup^{2,\infty}(\C)$. These two extensions agree on $\dot\W^{1,1}_{\bar\partial}(\C)\cap\dot\W^{1,1}_{\partial}(\C)=\dot\W^{1,1}(\C)$, since $\Con_c^\infty(\C)$ is dense in $\dot\W^{1,1}(\C)$ and convergence in $\dot\W^{1,1}(\C)$ implies convergence in both $\|\bar\partial f\|_{\Lup^1(\C)}$ and $\|\partial f\|_{\Lup^1(\C)}$. Hence $\T_\Omega(g+h)=\T_\Omega g+\T_\Omega h$ is well defined for $g\in\dot\W^{1,1}_{\bar\partial}(\C)$ and $h\in\dot\W^{1,1}_{\partial}(\C)$, and
\begin{multline*}
\|\T_\Omega(g+h)\|_{\Lup^{2,\infty}(\C)}\leq2\big(\|\T_\Omega g\|_{\Lup^{2,\infty}(\C)}+\|\T_\Omega h\|_{\Lup^{2,\infty}(\C)}\big)\\ \leq C\|\Omega\|_{\Lup^1(\mathbb S^1)}\big(\|\bar\partial g\|_{\Lup^1(\C)}+\|\partial h\|_{\Lup^1(\C)}\big).
\end{multline*}
Taking the infimum over all such decompositions $f=g+h$ gives the bounded extension $\T_\Omega:\dot\W^{1,1}_{\bar\partial}(\C)+\dot\W^{1,1}_{\partial}(\C)\ra\Lup^{2,\infty}(\C)$.
\end{proof}

\subsection*{Use of AI tools}
The authors used Claude Opus 5.5 (Anthropic) to explore examples related to the critical Lebesgue exponent $r=d$. The examples were suggested entirely by the AI. When these indicated that the sharp exponent should be $r=d-1$, we turned to the Poincar\'e inequality on the unit sphere and to the space $\dot \W^{-1,\infty}(\Sp)$. The approach to the planar case via complex analysis was proposed by the authors, but the AI assisted with the details and exposition of the planar argument in Section~\ref{plane}. The authors have independently verified, validated, and rewritten all parts of the paper that were influenced by LLM-generated material, and they take full responsibility for the mathematical content of the paper.

\end{document}